\documentclass[a4paper,12pt]{amsart}
\usepackage{amsmath,amsfonts,amssymb,amsthm}
\usepackage{graphics}
\usepackage{epsfig}
\usepackage{esint}
\usepackage{shuffle}
\usepackage{istgame}
\usepackage[hidelinks]{hyperref}

\usepackage{vmargin}
\setpapersize{A4}
\setmargins{2.5cm}{1.5cm}{16.5cm}{23.42cm}{10pt}{1cm}{0pt}{2cm}

\usepackage{enumerate}

\newcommand{\ignore}[1]{}

\newcommand{\E}{\mathbb{E}}
\DeclareMathOperator{\tr}{\text{tr}}

\newtheorem{definition}{Definition}[section]
\newtheorem{proposition}[definition]{Proposition}

\author[P.~Morfe]{Peter Morfe}
\email[P.~Morfe]{peter.morfe@mis.mpg.de}

\author[F.~Otto]{Felix Otto}
\email[F.~Otto]{felix.otto@mis.mpg.de}

\author[C.~Wagner]{Christian Wagner}
\email[C.~Wagner]{christian.wagner@mis.mpg.de}

\begin{document}

\title[Gaussian free-field as stream function: continuum version]{The Gaussian free-field as a stream function: continuum version of the scale-by-scale homogenization result}

\maketitle

\begin{abstract}
This note is
about a drift-diffusion process $X$ with a time-independent, divergence-free drift $b$,
where $b$ is a smooth Gaussian field that decorrelates over large scales.
In two space dimensions, this just fails to fall into the standard theory of stochastic
homogenization, and leads to a borderline super-diffusive behavior.
In a previous paper by Chatzigeorgiou, Morfe, Otto, and Wang (2022), precise asymptotics of the annealed second moments of $X$
were derived by characterizing the asymptotics of the effective diffusivity $\lambda_L$
in terms of an artificially introduced large-scale cut-off $L$.
The latter was carried out by a scale-by-scale homogenization, and implemented
by monitoring the corrector $\phi_L$ for geometrically increasing cut-off scales $L^+=ML$. 
In fact, proxies $(\tilde\phi_L,\tilde\sigma_L)$ for the corrector and flux corrector
were introduced incrementally and the residuum $f_L$ estimated.

\smallskip

In this short supplementary note, we reproduce the arguments of the above paper in the continuum setting
of $M\downarrow 1$. This has the advantage that the definition of the proxies
$(\tilde\phi_L,\tilde\sigma_L)$
becomes more transparent -- it is given by a simple It\^{o} SDE with $\ln L$ acting as
a time variable. It also has the advantage that the residuum $f_L$, which is a martingale,
can be efficiently and precisely estimated by It\^{o} calculus. 
This relies on the characterization of the
quadratic variation of the (infinite-dimensional) Gaussian driver.
\end{abstract}

\section{Definition of the proxies via It\^{o} SDEs}

For the set-up, bibliography, and the notation, we refer the reader to \cite{CMOW}.

\medskip

In this note we adopt a differential-geometric language and basis-free notation:
Since the generic coordinate direction $\xi$ acts as a linear form on $\mathbb{R}^2$,
i.~e. a {\it cotangent} vector, we combine the two (standard) scalar corrector
fields $\{\phi^i\}_{i=1,2}$ to a {\it tangent} vector field $\phi$, 
which amounts to\footnote{in this note, we use Einstein's notation in the strict form:
We add over repeated indices provided one is a super and the other a sub-script}
$\phi=\phi^ie_i$, where the vectors $\{e_i\}_{i=1,2}$ form the Cartesian basis.
Hence the contraction $\xi.\phi$ yields the corrector in direction $\xi$.
Thinking of $\nabla$ as the differential, which applied to a scalar field 
yields a one-form, that is, a cotangent vector field,
we interpret $\nabla\phi$ as a field of endomorphisms of cotangent space.
Interpreting also $a$ as such a field, the product $a\nabla\phi$ makes sense.

\medskip

Recall that the stream functions $\psi_L$'s are coupled such that the process
$[1,\infty)\ni L\mapsto\psi_L$ has independent increments. 
In view of $\mathbb{E}\psi_L^2=\ln L$, see \cite[(13)]{CMOW}, 
$[0,\infty)\ni\ln L=:s\mapsto\psi_L$ acts like Brownian motion,
with values in the space of 
(smooth) stationary Gaussian scalar fields.
In terms of this continuum variable $s$, the discrete relation \cite[(49)]{CMOW}
for the proxy $\tilde\lambda$ to the effective diffusivity turns into
\begin{align}\label{wr08}
d\tilde\lambda
=\frac{\varepsilon^2}{2\tilde\lambda}ds\quad\mbox{and}\quad\tilde\lambda|_{s=0}=1
\end{align}
in the limit $M \searrow 1$ (see \cite[(49)]{CMOW}).
In view of the Brownian character of $s\mapsto\psi_L$, 
its distributional derivative/infinitesimal increment $s\mapsto d\psi$ acts like 
(stationary, function valued, and Gaussian) white noise (in $s$).
In line with \cite[(38)]{CMOW}, we define the (stationary and vector-field valued)
white noises $d\phi$ and $d\sigma$ via the Helmholtz decomposition
\begin{align}\label{wr01}
\tilde\lambda \nabla d\phi+d\psi J=J\nabla d\sigma\quad\mbox{and}\quad
\mathbb{E}d\phi=\mathbb{E}d\sigma=0.
\end{align}
As discussed above, the white noises $\nabla d\phi$ and $\nabla d\sigma$ have values
in the space of fields of 
endomorphisms of cotangent space, so that provided we interpret $J$ as the 
counter-clockwise rotation by $\frac{\pi}{2}$ on (co)tangent space, 
(\ref{wr01}) makes sense.

\medskip

Equipped with the vectorial driver\footnote{in the jargon of SDEs} $d\phi$,
we now define the vectorial proxy corrector $\tilde\phi$ via an SDE in Itô form
that encapsulates the two-scale expansion:
\begin{align}\label{wr04}
d\tilde\phi=\varepsilon (1+\tilde\phi^i\partial_i)d\phi,\quad\tilde\phi|_{s=0}=0.
\end{align}
On the basis of the vectorial driver $d\sigma$ and the original 
scalar driver $d\psi$ we define the proxy flux corrector $\tilde\sigma$ via  
%
\begin{align}\label{wr02}
d\tilde\sigma= \varepsilon ( d\sigma+\tilde\sigma^i\partial_id\phi+d\psi\,\tilde\phi )
+ J\tilde\phi  \, \frac{\varepsilon^2}{2\tilde\lambda}ds, 
\quad\tilde\sigma|_{s=0}=0.
\end{align}
Definitions (\ref{wr04}) \& (\ref{wr02}) correspond to \cite[(36) \& (37)]{CMOW},
the motivation for (\ref{wr02}) lies in (\ref{wr03}) below.
Note that in view of the drift term in (\ref{wr02}), $\tilde\sigma$ is
not a martingale, as opposed to $\tilde\phi$. 

\medskip

As in \cite{CMOW}, we measure the quality of the proxies $\tilde\phi$ and $\tilde\sigma$
in terms of the residuum $f$ in
\begin{align}\label{wr07}
a({\rm id}+\nabla\tilde\phi)=\tilde\lambda\,{\rm id}+J\nabla\tilde\sigma+f,
\end{align}
where we recall the definition of $a$ from \cite[(11)]{CMOW}
\begin{align}\label{wr09}
a={\rm id}+\varepsilon\psi J\quad\mbox{with}\quad \psi|_{s=0}=0.
\end{align}
In this note, we give a self-contained proof for\footnote{what amounts to
a first-order error estimate in $\varepsilon$}

\begin{proposition}[{see \cite[(79)]{CMOW}}]\label{proposition}
%
	For $\varepsilon^2\ll 1$ it holds
	$$
	\E | f |^2 \lesssim \varepsilon^2 \tilde\lambda.
	$$
\end{proposition}

In preparation, we will establish in the next section 
that (\ref{wr07}) is consistent with the 
tensor-valued\footnote{Note that as a tensor product of a cotangent vector and a tangent vector,
the last term in (\ref{wr03}) canonically is an endomorphism of cotangent vector space.} SDE
%
\begin{align}\label{wr03}
df= \varepsilon\big( f \nabla d\phi+(\tilde\phi^ia-\tilde\sigma^iJ)\nabla\partial_id\phi
-(J\nabla d\psi)\otimes\tilde\phi\big) ,\quad f|_{s=0}=0,
\end{align}
see \cite[(50)]{CMOW}.
Note that \eqref{wr03} shows that the martingale $f$ has vanishing expectation, 
which was part of its definition in \cite[(48)]{CMOW}. This implies the exact consistency
$\mathbb{E}a({\rm id}+\nabla\tilde\phi)$ $=\tilde\lambda{\rm id}$,
which in \cite{CMOW} only holds in the limit $M\searrow1$, see \cite[(51)]{CMOW}.

\medskip

The approach of a continuum decomposition of an underlying Gaussian field,
often the Gaussian free field $\psi$ like here,
according to spatial scales $L$ is also used in quantum field theory
(where the underlying noise is thermal instead of quenched/environmental),
and is known as variance decomposition.
For instance, \cite{BG} use it to control (exponential) moments of the
field $\phi$ under consideration,
with help of suitable martingales in $\ln L$, analogous to $\tilde\phi$.
The variance decomposition also induces a
Hamilton-Jacobi-type evolution equation for the effective Hamiltonian in $\ln L$,
known as Polchinski equation. In \cite[Section 2.3]{GM}, to cite another recent
work, proxies for the solution of the Polchinski equation are constructed by truncation
of an expansion (analogous to an expansion in $\varepsilon$ here),
and the evolution of the residuum is monitored, like $f$ in (\ref{wr03}).

%
%
%
%


\section{Proof of (\ref{wr07}) by identification of quadratic variations}\label{sec:r-sde}

We show how the SDEs (\ref{wr04}), (\ref{wr02}), and (\ref{wr03}) imply the identity (\ref{wr07}). 
To this end, we monitor the tensor-field valued residuum
\begin{align}\label{wr12}
r := a({\rm id}+\nabla\tilde\phi)-\tilde\lambda\,{\rm id}-J\nabla\tilde\sigma-f
\end{align}
and shall establish the homogeneous SDE
\begin{align}\label{wr11}
dr= \varepsilon r\nabla d\phi,\quad r|_{s=0}=0.
\end{align}
The second item in (\ref{wr11}) follows directly from the second items in 
(\ref{wr08}), (\ref{wr04}), (\ref{wr02}), (\ref{wr03}), and (\ref{wr09}).

\medskip

In order to establish the first item in (\ref{wr11}), and in view of (\ref{wr12}), 
we need to identify the infinitesimal increment of the flux $a({\rm id}+\nabla\tilde\phi)$, 
which is a bilinear term in the drivers and thus involves a quadratic variation term. We claim that
%
\begin{align}\label{wr16}
d\big(a({\rm id}+\nabla\tilde\phi)\big)
=\varepsilon (-\tilde\lambda\nabla d\phi+J\nabla d\sigma)({\rm id}+\nabla\tilde\phi)
+a\nabla d\tilde\phi
+\varepsilon^2({\rm id}-J\nabla\tilde\phi J)\frac{ds}{2\tilde\lambda}.
\end{align}
Indeed, applying $\nabla$
to (\ref{wr04})
and appealing to Leibniz' rule in conjunction with 
the identity\footnote{where the latter is understood as a composition of endomorphisms} 
$(\nabla \tilde\phi^i) \otimes \partial_i d \phi$ $=\nabla\tilde\phi\nabla d\phi$, we obtain
\begin{align}\label{wr15}
d\nabla\tilde\phi=\varepsilon ({\rm id}+\nabla\tilde\phi)\nabla d\phi+ \varepsilon \tilde\phi^i\nabla\partial_id\phi.
\end{align}
In view of (\ref{wr09}), the main task is to identify the infinitesimal
increment of $\varepsilon\psi J\nabla\tilde\phi$;
by It\^{o} calculus this involves two covariations of the drivers:
\begin{align}\label{wr10}
d\big(\varepsilon \psi J\nabla\tilde\phi\big)
= \varepsilon d\psi J\nabla\tilde\phi+ \varepsilon \psi J\nabla d\tilde\phi
+\varepsilon^2 J({\rm id}+\nabla\tilde\phi)[d\psi \, \nabla d\phi]
+\varepsilon^2 \tilde\phi^i[d\psi \, \nabla\partial_id\phi].
\end{align}
These covariations of drivers are deterministic by the independence of increments
and constant in space by stationarity. Since $d\psi$ is invariant in law under point reflection,
$d\phi$ is (jointly) odd under the same operation in view of (\ref{wr01}),
and thus also $\nabla\partial_id\phi$. 
Hence the last contribution to (\ref{wr10}) vanishes,~i.~e.
\begin{align}\label{wr13}
[d\psi \, \nabla\partial_id\phi]=0.
\end{align}

\medskip

Thus the task is to identify the (deterministic) endomorphism $B$ (of cotangent space) in
\begin{equation}\label{cw14a}
[d\psi \, \nabla d\phi]=Bds.
\end{equation}
%
%
Since the scalar $d\psi$ is invariant in law under rotations, 
the endomorphism $\nabla d\phi$ is invariant in law under conjugation with rotations.
This transmits to $B$, which amounts to
\begin{align}\label{cw14}
B=\;\mbox{scalar} \; \times \; \mbox{rotation}.
\end{align}

\medskip
 
In order to identify $B$, we derive local relations between the drivers 
$d\phi$ and $d\psi$ from (\ref{wr01}):
\begin{align}\label{wr21}
{\rm tr}\nabla d\phi=0\quad\mbox{and}\quad
\tilde\lambda {\rm tr}\nabla d\phi J=d\psi.
\end{align}
To this purpose, we consider the $i$-th component of (\ref{wr01}), which
amounts to applying this endomorphism identity 
to the cotangent vector $e^i$ (where $\{e^i\}_{i=1,2}$ denotes 
the dual basis) and yields $\tilde\lambda\nabla d\phi^i+d\psi Je^i$
$=J\nabla d\sigma^i$. Applying the divergence\footnote{where for a cotangent vector $\xi$,
$\xi\cdot$ denotes the corresponding tangent vector} 
$\nabla\cdot= \partial_j e^j \cdot $ and using $\nabla\cdot\nabla=\triangle$ gives 
$\tilde\lambda\triangle d\phi^i+\partial_j d\psi(e^j\cdot Je^i)=0$.
In view of $d\phi^ie_i=d\phi$ and the skewness of $J$ in form of
$(e^j\cdot Je^i)e_i = - ( J e^j \cdot e^i ) e_i =-(Je^j)\cdot$, 
we obtain the tangent-vector identity
%
\begin{align}\label{wr30}
\tilde\lambda\triangle d\phi=
(J\nabla d\psi)\cdot.
\end{align}
Applying $\nabla=\partial_ie^i$, and noting that 
$e^i\otimes\partial_i\nabla d\psi\cdot$ coincides with the
Hessian endomorphism $\nabla^2d\psi$ of cotangent space, we obtain once more by skewness of $J$
the tensor-valued identity
\begin{align}\label{wr24}
\tilde\lambda\triangle \nabla d\phi=-\nabla^2 d\psi J.
\end{align}
Taking the trace and appealing to the symmetry of $\nabla^2 d\psi$ on the one hand, 
and multiplying with $J$ from the right, using $J^2=-{\rm id}$, 
and then taking the trace and appealing to
${\rm tr}\nabla^2=\triangle$ on the other hand gives 
\begin{align*}
\triangle {\rm tr}\nabla d\phi=0\quad\mbox{and}\quad
\triangle(\tilde\lambda{\rm tr}\nabla d\phi J-d\psi)=0.
\end{align*}
Since both expressions under $\triangle$ are stationary fields of vanishing expectation, 
this yields (\ref{wr21}).

\medskip

Because of the underlying white-noise type normalization 
\begin{align}\label{wr31}
[d\psi \, d\psi]=ds,
\end{align}
on the level of \eqref{cw14a}, and the locality\footnote{in differential geometric jargon} of the covariation seen as a bilinear form, (\ref{wr21}) translates into
\begin{align*}
{\rm tr}B=0\quad\mbox{and}\quad
\tilde\lambda{\rm tr}BJ=1,
\end{align*}
which in view of (\ref{cw14}) finally implies $B = - \frac{1}{2 \tilde\lambda} J$.
We thus have identified \eqref{cw14a}:
\begin{align}\label{wr14}
[d\psi\,\nabla d\phi]
=-J\frac{ds}{2\tilde\lambda}.
\end{align}
Inserting (\ref{wr13}) and (\ref{wr14}) into (\ref{wr10}) yields
\begin{align*}
d\big( \varepsilon \psi J\nabla\tilde\phi\big)
= \varepsilon d\psi J\nabla\tilde\phi+ \varepsilon \psi J\nabla d\tilde\phi
+\varepsilon^2({\rm id}-J\nabla\tilde\phi J)\frac{ ds }{2\tilde\lambda},
\end{align*}
which by (\ref{wr01}) and (\ref{wr09}) turns into (\ref{wr16}).

\medskip

Equipped with (\ref{wr16}) we finally turn to (\ref{wr11}).
We now argue that the first part of the drift term in (\ref{wr16}) is balanced by (\ref{wr08})
whereas the second part is compensated by (\ref{wr02}). Indeed, applying $\nabla$ to
(\ref{wr02}) we obtain because of $\nabla\tilde\sigma^i\otimes \partial_id\phi = \nabla\tilde\sigma \nabla d\phi$ and $ \nabla ( J \tilde\phi ) = \nabla \tilde\phi J^* $, where $ J^* $ denotes the adjoint endomorphism w.~r.~t.~the Euclidean inner product on the (co)tangent space,
\begin{align}\label{wr17}
dJ\nabla\tilde\sigma= \varepsilon ( J\nabla d\sigma+J\nabla\tilde\sigma\nabla d\phi+ 
\tilde\sigma^i J\nabla\partial_id\phi+J\nabla d\psi\otimes\tilde\phi+d\psi J\nabla\tilde\phi )
- \varepsilon^2 J\nabla\tilde\phi J \, \frac{d s}{2\tilde\lambda}.
\end{align}
Applying $d$ to (\ref{wr12}) and inserting (\ref{wr08}), (\ref{wr15}), (\ref{wr16}), 
and (\ref{wr17}) yields
\begin{align*}
dr & = -\varepsilon \tilde\lambda (\nabla d\phi)({\rm id} + \nabla\tilde\phi) + \varepsilon  J(\nabla d\sigma)\nabla\tilde\phi
+ \varepsilon a({\rm id}+\nabla\tilde\phi)\nabla d\phi\\
& + \varepsilon \tilde\phi^ia\nabla\partial_id\phi
- \varepsilon J\nabla\tilde\sigma\nabla d\phi- \varepsilon \tilde\sigma^i J \nabla\partial_id\phi
- \varepsilon J\nabla d\psi\otimes\tilde\phi- \varepsilon d\psi J\nabla\tilde\phi-df,
\end{align*}
which by definition (\ref{wr01}) of $(\nabla d\phi,\nabla d\sigma)$ in terms of $d\psi$
simplifies to
\begin{align*}
dr&=- \varepsilon \tilde\lambda\nabla d\phi
+ \varepsilon a({\rm id}+\nabla\tilde\phi)\nabla d\phi\\
& + \varepsilon \tilde\phi^ia\nabla\partial_id\phi
- \varepsilon J\nabla\tilde\sigma\nabla d\phi
- \varepsilon J\tilde\sigma^i\nabla\partial_id\phi
- \varepsilon J\nabla d\psi\otimes\tilde\phi-df.
\end{align*}
This motivates the form of (\ref{wr03}) that compensates the terms
that have an additional derivative on the drivers $\nabla d\phi$ and $d\psi$:
\begin{align*}
dr= \varepsilon ( -\tilde\lambda {\rm id}
+a({\rm id}+\nabla\tilde\phi)-J\nabla\tilde\sigma-f )\nabla d\phi,
\end{align*}
which by definition (\ref{wr12}) of $ r $ turns into the desired identity (\ref{wr11}).


\section{Proof of estimates by It\^{o} calculus}\label{sec:proofs}

Lemma 2 from \cite{CMOW} provides an iteration formula for $ f_{L+} $ that is used to estimate 
$ \E | f_L |^2 $. In the continuum approach, based on the SDE (\ref{wr03}), 
we use It\^{o} calculus to derive the differential inequality
\begin{equation}\label{add07}
d\mathbb{E}|f|^2 - \varepsilon^2  \mathbb{E} |f|^2 \frac{ ds }{2\tilde\lambda^2}
\leq \varepsilon^2  \E|\tilde\phi\otimes a-\tilde\sigma\otimes J|^2 \frac{ ds}{2L^2 \tilde\lambda^2} 
+ \varepsilon^2 \E|\tilde\phi|^2 \frac{2 ds}{L^2}.
\end{equation}
Here, $|f|^2$ stands for the squared Frobenius norm given by $|f|^2:=\tr f^* f $. Likewise, 
$|\tilde\phi \otimes a - \tilde\sigma \otimes J|^2$ denotes the squared 
Frobenius norm of that 3-tensor.

\begin{proof}[Proof of {\cite[Lemma 2]{CMOW}} revisited.]
Since $f$ is a martingale, its quadratic variation determines the evolution of 
$\mathbb{E}|f|^2$, the expectation of a quadratic expression in $f$. 
Because of the specific form of the latter, 
this involves a specific contraction of the 4-tensor describing all covariations,
%
\begin{align}\label{wr34}
d \E | f |^2 = \E [ \tr df^* df ],
\end{align}
which will be useful. Turning to (\ref{wr03}), we note that the parity argument
used for (\ref{wr13}) yields
\begin{align*}
[\nabla d\phi \, (\nabla\partial_i d\phi)^*]=0\quad\mbox{and}\quad
[\nabla d\phi \, (\nabla d\psi\otimes e_i)^*]=0.
\end{align*}
Hence the r.~h.~s.~of (\ref{wr34}) splits
\begin{align*}
\lefteqn{d\E|f|^2=\varepsilon^2 \E [ \tr ( f \nabla d \phi )^* ( f \nabla d \phi ) ]}\nonumber\\
&+\varepsilon^2 \E [  \tr ( (\tilde\phi^i a - \tilde\sigma^i J) 
\nabla \partial_i d \phi - {\tilde\phi^i}J\nabla d\psi\otimes{e_{i}} )^* 
( (\tilde\phi^{j} a - \tilde\sigma^{j}J) 
\nabla \partial_{j} d \phi - {\tilde\phi^j}J \nabla d\psi\otimes 
{e_{j}} ) ].
\end{align*}
By Cauchy-Schwarz applied to the inner product $\mathbb{E} [ {\rm tr} f^*g ]$, 
we may further split the last term:
\begin{align*}
d \E | f |^2
& \leq \varepsilon^2 \E [ \tr ( f \nabla d \phi )^* ( f \nabla d \phi ) ] + 2 \varepsilon^2 \E [ \tr ( (\tilde\phi^i a - \tilde\sigma^i J) \nabla \partial_i d \phi )^* 
( (\tilde\phi^{j} a - \tilde\sigma^{j} J) 
\nabla \partial_{j} d \phi ) ] \nonumber \\
& + 2 \varepsilon^2 \E [ \tr ( J \nabla d\psi \otimes \tilde\phi )^* ( J \nabla d\psi \otimes \tilde\phi ) ].
\end{align*}
Using the algebraic identities ${\rm tr}(ff')^*ff'={\rm tr}f^*ff'f'^*$,
$(\alpha^i{\rm id}+\beta^iJ)^*(\alpha^j{\rm id}+\beta^jJ)$ 
$=(\alpha^i\alpha^j+\beta^i\beta^j){\rm id}$ $+(\alpha^i\beta^j-\alpha^j\beta^i)J$, and 
${\rm tr}(\xi\otimes\phi)^*(\xi\otimes\phi)$ $=|\xi|^2|\phi|^2$, we isolate
the quadratic variations of the drivers:
\begin{align*}
d\E|f|^2 &\leq \varepsilon^2 \E\tr f^*f [\nabla d \phi(\nabla d \phi )^*] \\
&+2\varepsilon^2\E\big(\tilde\phi^i\tilde\phi^j
 +(\varepsilon\psi\tilde\phi^i-\tilde\sigma^i)(\varepsilon\psi\tilde\phi^j-\tilde\sigma^j)\big)
{\rm tr}[\nabla \partial_j d \phi (\nabla \partial_i d \phi )^*]
+2\varepsilon^2 \E |\tilde\phi|^2 [\nabla d\psi \cdot \nabla d\psi ],
\end{align*}
where in the second r.~h.~s.~term, the skew-symmetric part in $(i,j)$ drops
out because $[\nabla \partial_j d \phi (\nabla \partial_i d \phi )^*]$ is symmetric
in $(i,j)$ by stationarity\footnote{As a bilinear form the covariation satisfies a Leibniz rule. Since the covariation under consideration is deterministic and by stationarity constant in space, this implies that we can integrate by parts.}. 
By isotropy, the symmetric endomorphism 
$[\nabla d \phi(\nabla d \phi )^*]$ is a multiple of the identity;
likewise, ${\rm tr}[\nabla \partial_j d \phi (\nabla \partial_i d \phi )^*]$ is
a multiple of $\delta_{ij}$. Therefore, and using
\begin{equation}\label{fosw01}
2(|\tilde\phi|^2
+|\varepsilon\psi\tilde\phi-\tilde\sigma|^2) =|\tilde\phi\otimes a-\tilde\sigma\otimes J|^2,
\end{equation} 
the above inequality collapses to
\begin{align*}
d\E|f|^2 &\leq \frac{\varepsilon^2}{2} 
\E|f|^2 [{\rm tr}\nabla d \phi(\nabla d \phi )^*]\nonumber\\
&+\frac{\varepsilon^2}{2}\E|\tilde\phi\otimes a-\tilde\sigma\otimes J|^2
{\textstyle \sum_{i=1}^2}[{\rm tr}\nabla \partial_i d \phi (\nabla \partial_i d \phi )^*]
+2\varepsilon^2 \E |\tilde\phi|^2 [\nabla d\psi \cdot \nabla d\psi ].
\end{align*}
Hence by (\ref{wr08}) in form of 
\begin{align}\label{wr37}
d\tilde\lambda^2=\varepsilon^2ds,
\end{align}
(\ref{add07}) follows once we identify the following quadratic variations
%
\begin{align}
\tilde\lambda^2[{\rm tr}\nabla d\phi \, (\nabla d\phi)^*]&=ds,\label{wr18}\\
\tilde\lambda^2 L^2 {\textstyle \sum_{i=1}^2}
[{\rm tr}\nabla\partial_id\phi \, (\nabla\partial_id\phi)^*]
& =ds,\label{wr26}\\
L^2[\nabla d\psi\cdot\nabla d\psi] &=ds.\label{wr32}
\end{align}
\ignore{
that will be proven later on. {$e^i\otimes e^j$ did not make sense since
not an endomorphism of co-tangent space; also, one should not convert upper and lower indices
without a meaning}

\medskip

With help of (\ref{wr18}) we compute the first r.~h.~s.~term {of} (\ref{add04}):
$$
\varepsilon^2 \E [ \tr ( f \nabla d \phi )^* ( f \nabla d \phi ) ]
= \varepsilon^2 \E \tr f^* f [ \nabla d\phi (\nabla d\phi)^* ]
= \E |f|^2 \frac{d \tilde\lambda^2}{2 \tilde\lambda^2}.
$$
For the last two terms in (\ref{add04}) we use (\ref{wr26}) to obtain
$$
\varepsilon^2 \E [ \tr ( (\tilde\phi^i a - \tilde\sigma^i J) \nabla \partial_i d \phi )^* ( (\tilde\phi^i a - \tilde\sigma^i J) \nabla \partial_i d \phi ) ]
\lesssim \E | \tilde\phi \otimes a - \tilde\sigma \otimes J |^2 \frac{ d\tilde\lambda^2 }{\tilde\lambda^2 L^2}
$$
and
$$
\varepsilon^2 \E [ \tr ( J \nabla d\psi \otimes \tilde\phi )^* ( J \nabla d\psi \otimes \tilde\phi ) ]
= \varepsilon^2 \E | \tilde\phi |^2 [ J \nabla d \psi \cdot J \nabla d \psi ]
= \varepsilon^2 \E | \tilde\phi |^2 \frac{1}{L^2} [ d \psi \, d \psi ]
= \E | \tilde\phi |^2 \frac{d \tilde\lambda^2}{L^2}.
$$
The quadratic variation of $ \nabla d \psi $ may be computed explicitly from the covariance of $ \psi $. To avoid this computation one may use an alternative argument similar to (\ref{wr25prev}) and (\ref{wr25prev2}) that we present later on in the derivation of (\ref{wr25}). Together, the last three equations and (\ref{add04}) combine to (\ref{add07}).
}

\medskip

We start by deriving (\ref{wr18}) from (\ref{wr14}).
By stationarity we have $[{\rm tr}\nabla d\phi(\nabla d\phi)^*]$ $=[d\phi\cdot (-\triangle) d\phi]$,
into which we insert (\ref{wr30}) to the effect of\footnote{recall that $ \xi. \phi $ denotes the
canonical pairing of a cotangent vector $\xi$ and a tangent vector $\phi$}
$\tilde\lambda [{\rm tr}\nabla d\phi(\nabla d\phi)^*]$ $=[J\nabla d\psi.d\phi]$.
Once more by stationarity this turns into $\tilde\lambda [{\rm tr}\nabla d\phi\,(\nabla d\phi)^*]$ 
$={\rm tr}J[d\psi\,\nabla d\phi]$, 
so that it remains to appeal to (\ref{wr14}) and $J^2=-{\rm id}$.

\medskip

We now argue that (\ref{wr32}) follows from (\ref{wr31}). Indeed, by stationarity
we have $[\nabla d\psi\cdot\nabla d\psi]=[d\psi \, (-\triangle)d\psi]$.
By definition of the increment $d\psi$, it has values in the space of stationary
fields which are Fourier supported on $L|k|=1$, to the effect of $L^2(-\triangle)d\psi=d\psi$.
Hence we obtain the desired 
\begin{align}\label{wr35}
L^2[\nabla d\psi\cdot\nabla d\psi]=[d\psi\,d\psi].
\end{align}
\ignore{
By stationarity and isotropy, the symmetric matrix $[\nabla d\phi (\nabla d\phi)^*]$
is constant and a multiple of the identity; hence we have
\begin{align}\label{wr19}
\frac{1}{2}({\rm tr}A){\rm id} \, d\tilde\lambda^2 
= A \, d \tilde\lambda^2
:= \varepsilon^2 [\nabla d\phi (\nabla d\phi)^*].
\end{align}
We note that for any endomorphism $B$ we have
$ {\rm tr}BB^* $ $=({\rm tr}B)^2$ $+({\rm tr}BJ)^2$ $-2{\rm det}B$, and thus
\begin{align}\label{wr20}
{\rm tr}A \, d\tilde\lambda^2 = \varepsilon^2 [{\rm tr}\nabla d\phi \, {\rm tr}\nabla d\phi]
+ \varepsilon^2 [{\rm tr}\nabla d\phi J \, {\rm tr}\nabla d\phi J]
+ 2 \varepsilon^2 [\nabla d\phi^1 \cdot J\nabla d\phi^2].
\end{align}
Since $[\nabla d\phi^1 \cdot J\nabla d\phi^2]$ 
$=\partial_2[d\phi^1 \,  \partial_1d\phi^2] - \partial_1[d\phi^1 \, \partial_2d\phi^2]$
vanishes by stationarity, we infer (\ref{wr18}) from (\ref{wr19}) followed by (\ref{wr20})
into which we insert (\ref{wr21}).
}
Finally, we deduce (\ref{wr26}) from (\ref{wr18}) by a similar argument: By
stationarity, we have $\sum_{i=1}^2$
$[{\rm tr}\nabla\partial_id\phi (\nabla\partial_id\phi)^*]$
$=[{\rm tr}\nabla d\phi \, (-\triangle\nabla d\phi)^*]$, the Fourier support
of $d\psi$ transmits to $\nabla d\phi$ via the defining equation (\ref{wr01}).
\end{proof}


To estimate $ \E | f |^2 $, one needs the estimates on $ \E | \tilde\phi |^2 $, $ \E | \tilde\sigma |^2 $ and $ \E | \psi\tilde\phi |^2 $ that are stated in \cite[Lemma 4]{CMOW}. We now give their proofs in the continuum setting.

\begin{proof}[Proof of {\cite[Lemma 4]{CMOW}} revisited.]
\textit{Step 1: Evolution of $ \E | \tilde\phi |^2 $.} We claim that
\begin{equation}\label{adda01}
d \E |\tilde\phi|^2
- \varepsilon^2 \E | \tilde\phi |^2 \frac{ ds }{2\tilde\lambda^2}
= \varepsilon^2 \frac{ L^2 ds }{\tilde\lambda^2},
\end{equation}
which integrates to
\begin{equation}\label{adda02}
\E |\tilde\phi|^2 \lesssim \frac{\varepsilon^2 L^2}{\tilde\lambda^2}.
\end{equation}

\medskip

Let us first argue how (\ref{adda01}) implies (\ref{adda02}): We introduce the variables $ a := \E | \tilde\phi |^2 $ and $ x := \tilde\lambda^2 $ so that $ L^2 = \exp( \frac{2}{\varepsilon^2} ( x - 1 ) ) $, $ d \tilde\lambda^2 = \varepsilon^2 ds $ and (\ref{adda01}) reads
$$
\sqrt{x} \frac{d}{dx} \frac{a}{\sqrt{x}}
= \frac{da}{dx} - \frac{1}{2x} a
= \frac{ \exp ( \frac{2}{\varepsilon^2} ( x - 1 ) ) }{ x },
$$
which, since $ a(x = 1) = 0 $, implies
$$
a
= \frac{ \exp( \frac{2}{\varepsilon^2} ( x - 1 ) ) }{ x  }
\int_1^x dy \, ( \frac{ x }{ y } )^{\frac{3}{2}} \exp( \frac{2}{\varepsilon^2} ( y - x ) )
\lesssim \varepsilon^2 \frac{ \exp( \frac{2}{\varepsilon^2} ( x - 1 ) ) }{ x  }.
$$

\medskip

We now argue that (\ref{adda01}) holds. 
Since $\tilde\phi$ is a martingale we have
\begin{align*}
d\mathbb{E}|\tilde\phi|^2=\mathbb{E}[d\tilde\phi\cdot d \tilde\phi].
\end{align*}
Inserting (\ref{wr04}), the same parity argument that led to (\ref{wr13}) yields the splitting
\begin{align}\label{wr36}
d\mathbb{E}|\tilde\phi|^2
= \varepsilon^2 \E [ d\phi \cdot d\phi]
+ \varepsilon^2\mathbb{E} \tilde\phi^i \tilde\phi^j [ \partial_i d \phi \cdot \partial_j d \phi ].
\end{align}
Turning to the quadratic variation of the drivers, we have by isotropy
\begin{align}\label{wr18b}
2\tilde\lambda^2[\partial_i d \phi \cdot \partial_j d \phi ]
=\delta_{ij}\tilde\lambda^2[{\rm tr}\nabla d\phi(\nabla d\phi)^*] \stackrel{(\ref{wr18})}{=}
\delta_{ij}ds.
\end{align}
%
%
By the same argument that implied (\ref{wr35}) one infers
\begin{align}\label{wr18c}
\tilde\lambda^2[d\phi\cdot d\phi]=L^2\tilde\lambda^2[{\rm tr}\nabla d\phi(\nabla d\phi)^*]
\stackrel{(\ref{wr18})}{=}L^2ds.
\end{align} 
Inserting these two identities into (\ref{wr36}) and appealing to (\ref{wr37})
gives (\ref{adda01}).

\ignore{
we have
\begin{equation}\label{cov01}
\varepsilon^2 [ d\phi \cdot d \phi ]
= L^2 \varepsilon^2 \tr [ \nabla d \phi ( \nabla d \phi )^* ]
= \frac{L^2 d \tilde\lambda^2}{\tilde\lambda^2}.
\end{equation}
Together, the last two equations combine to
\begin{equation}\label{adda01a}
[ d \tilde\phi \cdot d \tilde\phi ]
= \frac{ L^2 d \tilde\lambda^2}{\tilde\lambda^2}
+ | \tilde\phi |^2 \frac{d\tilde\lambda^2}{2 \tilde\lambda^2},
\end{equation}
which implies (\ref{adda02}).
}

\medskip

\textit{Step 2: Evolution of $ \E | \tilde\sigma |^2 $.} We claim that
\begin{equation}\label{adda04}
d \E | \tilde\sigma |^2
= \varepsilon^2 \E | \tilde\sigma |^2 \frac{ds}{2 \tilde\lambda^2} 
+ 2 \varepsilon^2 \E \tilde\sigma \cdot J \tilde\phi \frac{ds}{\tilde\lambda}
+ \varepsilon^2 L^2 ds
+ \varepsilon^2 \E | \phi |^2 ds,
\end{equation}
which upon integration becomes
\begin{equation}\label{adda05}
\E | \tilde\sigma |^2 \lesssim \varepsilon^2 L^2.
\end{equation}
Indeed, using Young's inequality
and (\ref{adda02}) on (\ref{adda04}) implies
$$
d \E | \tilde\sigma |^2 - \varepsilon^2 \E | \tilde\sigma |^2 \frac{ds}{ \tilde\lambda^2 } \lesssim \varepsilon^2 L^2 ds,
$$
which in terms of $ b := \E | \tilde\sigma |^2 $ and (the previously introduced) $ x = \tilde\lambda^2 $ may be rewritten as
$$
x \frac{d}{dx} \frac{b}{x} = \frac{d b}{dx} - \frac{b}{x} \lesssim \exp( \frac{2}{\varepsilon^2} (x - 1) ),
$$
so that
$$
b \lesssim \exp( \frac{2}{\varepsilon^2} ( x - 1 ) ) \int_1^x dy \, \frac{x}{y} \exp( \frac{2}{\varepsilon^2} ( y - x ) ).
$$
(\ref{adda05}) follows from the last estimate.

\medskip

Now comes the argument for \eqref{adda04}. Note that by Itô's formula, we have
$$
d | \tilde\sigma |^2 = 2 \tilde\sigma \cdot d \tilde\sigma + [ d \tilde\sigma \cdot d \tilde\sigma ]
$$
so that by (\ref{wr02}), in which we substitute (\ref{wr37}), we obtain
$$
d \E | \tilde\sigma |^2 = 2 \varepsilon^2 \E \tilde\sigma \cdot J \tilde\phi \frac{ds}{\tilde\lambda} + \E [ d \tilde\sigma \cdot d \tilde\sigma ].
$$
Let us observe that a parity argument similar to (\ref{wr13}) implies
$$
[ d\tilde\sigma \cdot d\tilde\sigma ]
= \varepsilon^2 \tilde\sigma^i \tilde\sigma^j [ \partial_i d\phi \cdot \partial_j d\phi ]
+ 2 \varepsilon^2 \tilde\sigma^i \tilde\phi \cdot [ d \psi \, \partial_i d \phi ]
+ \varepsilon^2 [ ( d\sigma + d\psi \, \tilde\phi ) \cdot ( d\sigma + d\psi \, \tilde\phi ) ].
$$
Using (\ref{wr14}), (\ref{wr37}) and (\ref{wr18b}) this equation simplifies to
$$
[ d\tilde\sigma \cdot d\tilde\sigma ]
= \varepsilon^2 | \tilde\sigma |^2 \frac{ds}{2 \tilde\lambda^2} 
+ \varepsilon^2 \tilde\sigma \cdot J \tilde\phi \frac{ds}{\tilde\lambda}
+ \varepsilon^2 [ d\sigma \cdot d \sigma ]
+ 2 \varepsilon^2 \tilde\phi \cdot [ d\psi \, d\sigma ]
+ \varepsilon^2 | \tilde\phi |^2 [ d \psi \, d \psi ].
$$
Since $ \tilde\phi $ has vanishing expectation, the mixed term $ \tilde\phi \cdot [ d \psi \, d \sigma ] $ has vanishing expectation, so that overall
$$
d \E | \tilde\sigma |^2
= \varepsilon^2 \E | \tilde\sigma |^2 \frac{ds}{2 \tilde\lambda^2} 
+ 2 \varepsilon^2 \E \tilde\sigma \cdot J \tilde\phi \frac{ds}{\tilde\lambda}
+ \varepsilon^2 \E [ d\sigma \cdot d\sigma ]
+ \varepsilon^2 \E | \tilde\phi |^2 ds.
$$
Below we will argue that
\begin{equation}\label{fosw02}
\varepsilon^2 [d \sigma \cdot d \sigma] = \varepsilon^2 L^2 ds,
\end{equation} 
which together with the previous equation implies (\ref{adda04}).

\medskip

We now give the argument for \eqref{fosw02}. Via an integration by parts we obtain $ \E \nabla d \phi^i \cdot J \nabla d \sigma^i = 0 $ so that in particular $ \tr \E \nabla d \phi ( J \nabla d \sigma )^* = 0 $, which means that $ \nabla d \phi $ and $ J \nabla d \sigma $ are uncorrelated. Hence, (\ref{wr01}) provides an orthogonal decomposition of $ d\psi J $. Therefore, and since all of the following expressions are deterministic, we may conclude
$$
[ \tr d \psi J ( d \psi J )^* ]
= \tilde\lambda^2 [ \tr \nabla d \phi ( \nabla d \phi )^* ] + [ \tr J \nabla d \sigma ( J \nabla d \sigma )^* ].
$$
In view of (\ref{wr31}), and (\ref{wr18}) this implies
$$
[ \tr \nabla d \sigma ( \nabla d \sigma )^* ]
= 2 [ d \psi \, d \psi ] - \tilde\lambda^2 [ \tr \nabla d \phi ( \nabla d \phi )^* ]
=  ds
$$
The argument that leads to (\ref{wr35}) now implies (\ref{fosw02}).

\medskip

\textit{Step 3: Evolution of $ \E | \tilde\phi |^ 4 $.} We claim that
\begin{equation}\label{eqn4m}
d \E | \tilde\phi |^4
\lesssim \varepsilon^2 \E | \tilde\phi |^4 \frac{ds}{ \tilde\lambda^2}
+ \varepsilon^2 \E | \tilde\phi |^2 \frac{ L^2 ds}{\tilde\lambda^2},
\end{equation}
which integrates to
\begin{equation}\label{eqn4mb}
\E | \tilde\phi |^4 \lesssim \varepsilon^4 \frac{L^4}{\tilde\lambda^4}.
\end{equation}
The integration follows along the lines of (\ref{adda05}).

\medskip

Itô's formula implies
$$
d | \tilde\phi |^2 = 2 \tilde\phi \cdot d \tilde\phi + [ \tilde\phi \cdot \tilde\phi ],
$$
which we  combine with (\ref{wr04}) and (\ref{wr18b}), (\ref{wr18c}) on the quadratic variation to obtain
$$
d | \tilde\phi |^2 
= 2 \varepsilon \tilde\phi \cdot d \phi
+ 2 \varepsilon \tilde\phi^i \tilde\phi \cdot \partial_i d \phi 
+ \varepsilon^2 \frac{ L^2 ds }{\tilde\lambda^2}
+ \varepsilon^2 | \tilde\phi |^2 \frac{ds}{2 \tilde\lambda^2}.
$$
Since the covariation of $ \phi $ and $ \partial_i \phi $ vanishes, we obtain
$$
\begin{aligned}
d | \tilde\phi |^4
&= \text{martingale}
+ \varepsilon^2 | \tilde\phi|^4 \frac{ds}{ \tilde\lambda^2}
+ 2 \varepsilon^2 | \tilde\phi|^2 \frac{L^2 ds}{ \tilde\lambda^2} \\
& + 4 \varepsilon^2 [ ( \tilde\phi \cdot d \phi ) ( \tilde\phi \cdot d \phi ) ]
+ 4 \varepsilon^2 [ ( \tilde\phi^i \tilde\phi \cdot \partial_i d \phi ) ( \tilde\phi^j \tilde\phi \cdot \partial_j d \phi ) ],
\end{aligned}
$$
which implies
$$
\begin{aligned}
d \E | \tilde\phi |^4
&= \varepsilon^2 \E | \tilde\phi |^4 \frac{ds}{ \tilde\lambda^2}
+ 2 \varepsilon^2 \E | \tilde\phi |^2 \frac{ L^2 ds}{\tilde\lambda^2} \\
& + 4 \varepsilon^2 \E [ ( \tilde \phi \cdot d \phi ) ( \tilde\phi \cdot d \phi ) ]
+ 4 \varepsilon^2 \E [ ( \tilde\phi^i \tilde\phi \cdot \partial_i d \phi ) ( \tilde\phi^j \tilde\phi \cdot \partial_j d \phi ) ].
\end{aligned}
$$
For the third term, we use (\ref{wr37}) and (\ref{wr18b}), which yields
$$
4 \varepsilon^2 \E [ ( \tilde \phi \cdot d \phi ) ( \tilde\phi \cdot d \phi ) ]
= 2 \varepsilon^2 \E | \tilde\phi |^2 \frac{ L^2 ds }{ \tilde\lambda^2 }.
$$
For the fourth term, we note that the full covariation of first derivatives can be estimated by Cauchy-Schwarz, i.~e.
$$
\E [ \partial_i d \phi^k \, \partial_j d \phi^l ] 
\leq \E \tr [ ( \nabla d \phi )^* ( \nabla d \phi ) ].
$$
By a contraction we obtain the estimate
$$
\E [ ( \tilde\phi^i \tilde\phi \cdot \partial_i d \phi ) ( \tilde\phi^j \tilde\phi \cdot \partial_j d \phi ) ]
\lesssim \E | \tilde\phi |^4 \tr [ ( \nabla d \phi )^* ( \nabla d \phi ) ].
$$
By using (\ref{wr18}) this yields (\ref{eqn4m}).
\end{proof}

We collected all ingredients to estimate $ \E | f |^2 $,~i.~e.~provide an alternative proof of \cite[Lemma 5]{CMOW}.

\begin{proof}[Proof of Proposition \ref{proposition}.]
We are now ready to prove
\begin{equation}\label{adda10}
\E |f|^2 \lesssim \varepsilon^2 \tilde\lambda
\end{equation}
in the continuum setting.
Indeed, (\ref{add07}), together with (\ref{fosw01}) and (\ref{adda02}), implies
\begin{align*}
d\mathbb{E} | f |^2 - \varepsilon^2 \mathbb{E} | f |^2 \frac{ ds }{ 2 \tilde\lambda^2 }
\lesssim \varepsilon^2 \E | \varepsilon \psi \tilde\phi - \tilde\sigma |^2 \frac{ ds }{ L^2 \tilde\lambda^2 } 
+ \varepsilon^4 \frac{ ds }{ \tilde\lambda^2 }.
\end{align*}
By using Cauchy-Schwarz on the product $ | \tilde\phi \psi |^2 $ together with (\ref{adda05}) and (\ref{eqn4mb}) the above becomes
$$
d \E|f|^2 - \varepsilon^2 \E |f|^2 \frac{ds}{2\tilde\lambda^2}
\lesssim \varepsilon^6 ( \E \psi^4 )^{ \frac{1}{2} } \frac{ ds }{ \tilde\lambda^4 } + \varepsilon^4 \frac{ ds }{ \tilde\lambda^2 }.
$$
The Gaussianity of $ \psi $ implies $ ( \E \psi^4 )^{ \frac{1}{2} } \sim \E \psi^2 = \log L $, which together with (\ref{wr08}) in form of $ \tilde\lambda^2 = 1 + \varepsilon^2 \log L $, shows that the last term on the r.~h.~s.~is the dominant term so that
$$
d \E|f|^2 - \varepsilon^2 \E |f|^2 \frac{ds}{2\tilde\lambda^2}
\lesssim \varepsilon^4 \frac{ds}{\tilde\lambda^2}.
$$
In terms of $ c := \E |f|^2 $ and $ x = \tilde\lambda^2 $ this may be rewritten in the form
$$
\sqrt{x} \frac{d}{dx} \frac{c}{\sqrt{x}} \lesssim \varepsilon^2 \frac{1}{x}
$$
so that
$$
c \lesssim \varepsilon^2 \sqrt{x} \int_1^x dy \, \frac{1}{y^{\frac{3}{2}}}.
$$
This is nothing but (\ref{adda10}) in terms of $ c $ and $ x $.
\end{proof}

\section*{Acknowledgements}

The authors thank Lihan Wang for fruitful discussions in earlier stages of this work.

\end{document}